\documentclass[a4paper,11pt]{article} 
\usepackage[italian,english]{babel}			
\usepackage{amsmath,amssymb,amsthm}
\usepackage{amsfonts} 
\usepackage{graphicx} 
\usepackage{braket}
\usepackage{times}				
\usepackage{ifthen}
\usepackage{xspace}
\usepackage{fancyhdr}
\usepackage{mathrsfs}
\usepackage{enumerate}
\usepackage{verbatim}
\usepackage{hyperref}
\usepackage{mathrsfs}
\usepackage[latin1]{inputenc}
\usepackage{color} 

\definecolor{Red}{cmyk}{0,1,1,0.2}

\usepackage{fancyhdr}

 \def\hgi{\widehat{\gamma}_i}
\def\bg{\overline{\gamma}}
\usepackage{fancyhdr}

\newtheorem{definition}{Definition}[section]

\theoremstyle{definition}
\newtheorem{example}{Example}[section]
\theoremstyle{remark}
\newtheorem{remark}{Remark}[section]
\theoremstyle{plain}
\newtheorem{theorem}{Theorem}[section]
\newtheorem{lemma}{Lemma}[section]
\newtheorem{proposition}{Proposition}[section]
\newtheorem{corollary}{Corollary}[section]

\numberwithin{equation}{section}

\title{\bf Existence and uniqueness for Mean Field Games\\ with state constraints\footnote{This work was partly supported by the University of Rome ``Tor Vergata" (Consolidate the Foundations 2015) and by the Istituto Nazionale di Alta Matematica ``F. Severi'' (GNAMPA 2016 Research Projects). The second author is grateful to the Università Italo Francese (Vinci Project 2015).}}
\author{{\sc Piermarco Cannarsa}\footnote{%
		Dipartimento di Matematica, Universit\`{a}  di Roma ``Tor Vergata" - {\tt cannarsa@mat.uniroma2.it}}\\
	{\sc Rossana Capuani}\footnote{%
		Dipartimento di Matematica, Universit\`{a} di Roma ``Tor Vergata" and CEREMADE, Universit\'e Paris-Dauphine- {\tt capuani@mat.uniroma2.it}}
			\\}
\date{}
\begin{document}
	\maketitle
	\begin{abstract}
		\noindent
		In this paper, we study deterministic mean field games for agents who operate in a bounded domain. In this case, the existence and uniqueness of Nash equilibria cannot be deduced as for unrestricted state space because, for a large set of initial conditions, the uniqueness of the solution to the associated minimization problem is no longer guaranteed. We attack the problem by interpreting equilibria as measures in a space of arcs. In such a relaxed environment the existence of solutions follows by set-valued fixed point arguments. Then, we give a uniqueness result for such equilibria under a classical monotonicity assumption.
	\end{abstract}
	\noindent \textit{Keywords}: Mean field games, Nash equilibrium, state constraints, Hamilton-Jacobi-Bellman equations. \\
	\noindent \textbf{MSC Subject classifications}: 49J15, 49J30, 49J53, 49N90.\\
\section{Introduction}
Mean field games (MFG) theory has been introduced simultaneously by Lasry and Lions (\cite{8}, \cite{9}, \cite{10}) and by Huang, Malham\'{e} and Caine (\cite{h1}, \cite{h2}) in order to study large population differential games. 
The main idea of such a theory is to borrow from statistical physics the general principle of a mean-field approach to describe equilibria in a system of many interacting particles. 

In game theory, for a system with a finite number of players, the natural notion of equilibrium is the one introduced by John Nash. So, the notion of mean-field equilibrium suggested by Lasry-Lions is justified as being the limit, as $N\rightarrow\infty$, of the Nash equilibria for $N$-player games, under the assumption that players are symmetric and rational. 

In  deterministic settings, the equilibrium found in the mean field limit turns out to be a solution of the  forward-backward system of PDEs  
\begin{equation}\label{mfg}
(MFG)
\begin{cases}
-\partial_t u +H(x,Du)=F(x,m) \ \ \ \mbox{in} \ [0,T]\times {\Omega},\\
\partial_t m-div(mD_pH(x,Du))=0\ \ \ \mbox{in} \ [0,T]\times {\Omega},\\
m(0)=m_0 \ \ \ \ u(x,T)=G(x,m(T))
\end{cases}
\end{equation}
which couples a Hamilton-Jacobi-Bellman equation (for the value function $u$ of the generic player) with a continuity equation (for the density $m$ of players). Here $\Omega\subset \mathbb{R}^n$ represents the domain in the state space in which agents are supposed to operate.

The well-posedness of system \eqref{mfg} was developed for special geometries of the domain $\Omega$, namely when $\Omega$ equals the flat torus $\mathbb{T}^n=\mathbb{R}^n/\mathbb{Z}^n$, or the whole space $\mathbb{R}^n$ (see, e.g., \cite{34}, \cite{9}, \cite{10}).
The goal of the present paper is to study the well-posedness of the MFG problem subject to state constraints, that is, when players are confined into a compact domain $\overline{\Omega}\subseteq \mathbb{R}^n$. 

In the above references, the solution of $\eqref{mfg}$ on $[0,T]\times \mathbb{T}^n$ is obtained by a fixed point argument which uses in an essential way the fact that viscosity solutions of the Hamilton-Jacobi equation
\begin{equation*}
-\partial_t u +H(x,Du)=F(x,m)\quad \mbox{in} \ [0,T]\times \mathbb{T}^n
\end{equation*}
are smooth on a sufficiently large set to allow the continuity equation
\begin{equation*}
\partial_t m-div(mD_pH(x,Du))=0 \quad \mbox{in} \ [0,T]\times \mathbb{T}^n
\end{equation*}
to be solvable. Specifically, it is known that $u$ is of class $C^{1,1}_{loc}$ outside a closed singular set of zero Lebesgue measure. In this way, the coefficient $D_pH(x,Du)$ in the continuity equation turns out to be locally Lipschitz continuous on a ``sufficiently large'' open set. Such an ``almost smooth'' structure is lost in the presence of  state constraints \cite[ Example~1.1]{35}.
Therefore, in order to prove the existence of solutions to \eqref{mfg} a complete change of paradigm is necessary. 

In this paper, following the Lagrangian formulation of the unconstrained  MFG problem proposed in \cite{cc}, we define a ``relaxed'' notion of constrained MFG equilibria and solutions, for which we give existence and uniqueness results. Such a formulation consists of replacing probability measures on $\overline{\Omega}$ with measures on arcs in $\overline{\Omega}$.
More precisely, on the metric space
\begin{equation*}
\Gamma=\Big\{\gamma\in AC(0,T;\mathbb{R}^n): \gamma(t)\in\overline{\Omega}, \ \ \ \forall t\in[0,T] \Big\}
\end{equation*}
with the uniform metric, for any $t\in [0,T]$ we consider the evaluation map  $e_t:\Gamma\to \overline \Omega$ defined by 
$$
e_t(\gamma)= \gamma(t)\qquad (\gamma\in\Gamma).
$$
Given any probability measure $m_0$ on $\overline{\Omega}$, we denote by ${\mathcal P}_{m_0}(\Gamma)$ the set of all Borel probability measures $\eta$ on $\Gamma$ such that $e_0\sharp \eta=m_0$ and we consider,
for any $\eta \in \mathcal{P}_{m_0}(\Gamma)$, 
the functional
\begin{equation}
J_\eta [\gamma]=\int_0^T \Big[L(\gamma(t),\dot \gamma(t))+ F(\gamma(t),e_t\sharp\eta)\Big]\ dt + G(\gamma(T),e_T\sharp\eta)
\qquad(\gamma\in\Gamma).
\end{equation}
Then, we call a measure $\eta\in\mathcal{P}_{m_0}(\Gamma)$ a {\em constrained MFG equilibrium} for $m_0$ if $\eta$ is supported on the set of all curves $\overline{\gamma}\in \Gamma$ such that
\begin{equation*}
J_\eta[\overline{\gamma}]\leq J_\eta[\gamma]  \ \ \forall \gamma\in\Gamma, \ \gamma(0)=\overline{\gamma}(0).
\end{equation*}
Thus, we  obtain the existence of constrained MFG equilibria for $m_0$ (Theorem \ref{tesistenza}) by applying the Kakutani fixed point theorem \cite{ka}. At this point, it is natural to define a {\em mild solution of the constrained MFG problem} in $\overline{\Omega}$ as a pair $(u,m)\in C([0,T]\times \overline{\Omega})\times C([0,T]; \mathcal{P}(\overline{\Omega}))$, where 
$m$ is given by
$m(t)=e_t\sharp \eta$
for some constrained MFG equilibrium $\eta$ for $m_0$ and
\begin{equation*}
u(t,x)=  \inf_{\tiny\begin{array}{c}
	\gamma\in \Gamma\\
	\gamma(t)=x
	\end{array}} 
\Big\{\int_t^T \Big[L(\gamma(s),\dot \gamma(s))+ F(\gamma(s), m(s))\Big]\ ds + G(\gamma(T),m(T))\Big\}.
\end{equation*} 
In this way, the existence of mild solutions of the constrained MFG problem in $\overline{\Omega}$ (Corollary \ref{cesistenza}) becomes an easy corollary of the existence of equilibria for $m_0$ (Theorem \ref{tesistenza}), whereas the uniqueness issue for such a problem remains a more challenging question.  As observed by Lasry and Lions, in absence of state constraints uniqueness can be addressed by imposing suitable monotonicity assumptions on the data. We show that the same general strategy can be adopted even for constrained problems (Theorem \ref{u}). However, we have to interpret the Lasry-Lions method differently because, as recalled above, solutions are highly nonsmooth in our case.

The results of this paper can be regarded as an initial step of the study of deterministic MFG systems with state constraints. The natural sequel of our analysis would be to show that mild solutions to the constrained MFG problem in $\overline{\Omega}$ satisfy the MFG system in a suitable point-wise sense and, possibly, derive the uniqueness of solutions from such a system.

This paper is organised as follows. In Section 2, we introduce the notation and recall preliminary results. In Section 3, we define constrained MFG equilibria and we prove their existence. Section 4 is devoted to the study of mild solutions of the constrained MFG problem, in particular to the uniqueness issue. 
	\section{Preliminaries}
	\subsection{Notation}
	Throughout this paper we denote by $|\cdot|$, $\langle  \cdot  \rangle$ , respectively, the Euclidean norm and scalar
	product in $\mathbb{R}^n$. For any subset $S \subset \mathbb{R}^n$, $\overline{S}$ stands for its closure, $\partial S$ for its boundary and $S^c = \mathbb{R}^n \setminus S$ for the complement of $S$. We denote by $\mathbf{1}_{S}:\mathbb{R}^n\rightarrow \{0,1\}$ the characteristic function of $S$, i.e.,
	\begin{align*}
	\mathbf{1}_{S}(x)=
	\begin{cases}
	1  \ \ \ &x\in S,\\
	0 &x\in S^c.
	\end{cases}
	\end{align*}
	 We write $AC(0,T;\mathbb{R}^n)$ for the space of all absolutely continuous $\mathbb{R}^n$-valued functions on $[0,T]$, equipped with the uniform metric. We observe that such a space is not complete.\\
	 For any measurable function $f:[0,T]\rightarrow \mathbb{R}^n$, we set
	 \begin{equation*}
	 ||f||_2=\left(\int_0^T |f|^2\,dt\right)^{\frac{1}{2}}.
	 \end{equation*}
 	Let $\Omega \subset \mathbb{R}^n$ be a bounded open set with $\mathcal{C}^2$ boundary.
	The distance function from $\overline{\Omega}$ is the function $d_\Omega :\mathbb{R}^n \rightarrow [0,+ \infty[$ defined by
	\begin{equation*}
	d_\Omega(x):= \inf_{y \in \overline{\Omega}} |x-y| \ \ \ \ \ (x\in\mathbb{R}^n).
	\end{equation*}
 We define the oriented boundary distance from $\partial \Omega$ by
\begin{equation}
b_\Omega(x)=d_\Omega(x) -d_{\Omega^c}(x) \ \ \ \ (x\in\mathbb{R}^n).
\end{equation}
We recall that, since the boundary of $\Omega$ is of class $\mathcal{C}^2$, there exists $\rho_0>0$ such that
\begin{equation}\label{dn}
b_\Omega(\cdot)\in C^2_b \ \ \text{on} \ \ \partial\Omega +B_{\rho_0}=\Big\{y\in B(x,\rho_0): x\in \partial\Omega\Big\},
\end{equation}
where $C^2_b$ is the set of all functions with bounded derivates of first and second order. Throughout the paper, we suppose that $\rho_0$ is fixed so that \eqref{dn} holds.\\

	\subsection{Results from Measure Theory}
	In this section we introduce, without proof, some basic tools needed in the paper (see, e.g., \cite{5}).\\
	Let $X$ be a separable metric space, we denote by $\mathscr{B}(X)$ the family of the Borel subset of $X$ and by $\mathcal{P}(X)$ the family of all Borel probability measures on $X$.
	The support of $\mu \in \mathcal{P}(X)$, $supp(\mu)$, is the closed set defined by
	\begin{equation}
	supp (\mu) := \Big \{x \in X: \mu(V)>0\ \mbox{for each neighborhood V of $x$}\Big\}.
	\end{equation}
	We say that a sequence $(\mu_n)\subset \mathcal{P}(X)$ is narrowly convergent to $\mu \in \mathcal{P}(X)$ if
	\begin{equation*}
	\lim_{n\rightarrow \infty} \int_X f(x)\,d\mu_n(x)=\int_X f(x) \,d\mu(x) \ \ \ \ \forall f \in C^0_b(X),
	\end{equation*}
	where $C^0_b(X)$ is the set of all bounded continuous functions on $X$.\\
	We recall an interesting link between narrow convergence of probability measures and Kuratowski convergence of their supports.
	\begin{proposition}\label{p21}
		If $(\mu_n)\subset \mathcal{P}(X)$ is a sequence narrowly converging to $\mu \in \mathcal{P}(X)$ then $supp(\mu) \subset K-\liminf_{n\rightarrow \infty} supp(\mu_n)$, i.e.
		\begin{equation*}
		\forall x\in supp(\mu) \ \exists \ x_n \in supp(\mu_n) : \lim_{n\rightarrow \infty} x_n=x.
	\end{equation*}
	\end{proposition}
	\noindent
	The following theorem is a useful characterization of relatively compact sets with respect to narrow topology.
	\begin{theorem}{(Prokhorov's Theorem })\label{tp}
	If a set $\mathcal{K} \subset \mathcal{P}(X)$ is tight, i.e.
	\begin{equation*}
	\forall \epsilon>0 \ \exists \ K_\epsilon \ compact \ in \ X \ such \ that \ \widehat \eta (K_\epsilon)\geq 1-\epsilon \ \ \forall \widehat \eta \in \mathcal{K},
	\end{equation*}
	then $\mathcal{K}$ is relatively compact in $\mathcal {P}(X)$ with respect to narrow topology. Conversely, if $X$ is a separable complete metric space then every relatively compact subset of $\mathcal {P}(X)$ is tight.
	\end{theorem}
Let $X$ be a separable metric space. We recall that $X$ is a Radon space if every Borel probability measure $\mu\in \mathcal{P}(X)$ satisfies
	\begin{equation*}
	\forall B\in\mathscr{B}(X), \forall \epsilon>0, \ \exists K_\epsilon \ \mbox{compact}\ \mbox{with}\  K_\epsilon\Subset B \ \mbox{such that}\ \mu(B\setminus K_\epsilon)\leq \epsilon.
	\end{equation*}
Let us denote by $d$ the distance on $X$ and, for $p\in[1,+\infty)$, by $\mathcal{P}_p(X)$ the set of probability measures $m$ on $X$ such that
	\begin{equation*}
	\int_X d^p(x_0,x)\,dm(x) <+\infty, \ \ \ \ \forall x_0 \in X.
	\end{equation*}
	The \textit{Monge-Kantorowich distance} on $\mathcal{P}_p(X)$ is given by
	\begin{equation}\label{dis1}
	d_p(m,m')=\inf_{\lambda \in\Pi(m,m')}\Big[\int_{X^2}d(x,y)^p\,d\lambda(x,y) \Big ]^{1/p},
	\end{equation}
	where $\Pi(m,m')$ is the set of Borel probability measures on $X\times X$ such that $\lambda(A\times \mathbb{R}^n)=m(A)$ and $\lambda(\mathbb{R}^n\times A)=m'(A)$ for any Borel set $A\subset X$. In the particular case when $p=1$, the distance $d_p$ takes the name of Kantorovich-Rubinstein distance and the following formula holds
	\begin{equation}
	d_1(m,m')=sup\Big\{\int_X f(x)\,dm(x)-\int_X f(x)\,dm'(x) \ |\ f:X\rightarrow\mathbb{R} \ \ \mbox{is 1-Lipschitz}\Big\}, 
	\end{equation}
	for all $m$, $m'\in\mathcal{P}_1(X)$.
	In the next result, we recall the relationship between the weak-$\ast$ convergence of measures and convergence with respect to $d_p$. 
	\begin{proposition}\label{cm}
	If a sequence of measures $\{\mu_n\}_{n\geq 1} \subset\mathcal{P}_p(X)$ converges to $\mu$ for $d_p$, then $\{\mu_n\}_{n\geq 1}$ weakly converges to $\mu$.
	"Conversely", if $\mu_n$ is concentrated on a fixed compact subset of $X$ for all $n\geq 1$ and $\{\mu_n\}_{n\geq 1}$ weakly converges to $\mu$, then the $\{\mu_n\}_{n\geq 1}$ converges to $\mu$ in $d_p$.
	\end{proposition}
Given separable metric spaces $X_1$ and $X_2$ and a Borel map $f: X_1 \rightarrow X_2$, we recall that the push-forward of a measure $\mu \in \mathcal{P}(X_1)$ through $f$ is defined by
		\begin{equation}
		f \sharp \mu(B):=\mu(f^{-1}(B)) \ \ \ \forall B \in \mathscr{B}(X_2).
		\end{equation} 
	\noindent
	The push-forward is characterized by the fact that
	\begin{equation}
	\int_{X_1}g(f(x))\,d\mu(x)=\int_{X_2}g(y)\,df\sharp\mu(y)
	\end{equation}
	for every Borel function $g:X_2 \rightarrow\mathbb{R}$.\\
We conclude this preliminary session by recalling the disintegration theorem.
	\begin{theorem}\label{td}
		Let $X$, $Y$ be Radon separable metric spaces, $\mu \in \mathcal{P}(X)$, let $\pi: X \rightarrow Y$ be a Borel map and let $\eta=\pi \sharp \mu \in \mathcal{P}(Y)$. Then there exists an $\eta$-a.e. uniquely determined Borel measurable family \footnote{We say that $\{\mu_y\}_{y\in Y}$ is a Borel family (of probability measures) if $y\in Y\longmapsto \mu_y(B)\in \mathbb{R}$ is Borel for any Borel set $B\subset X$.} of probabilities $\{\mu_y\}_{y\in Y}\subset \mathcal{P}(X)$ such that
		\begin{equation}
		\mu_y(X\setminus\pi^{-1}(y))=0 \ \ \mbox{for $\eta$-a.e.} \ y \in Y
		\end{equation}
		and
		\begin{equation}
		\int_{X} f(x)\,d\mu(x)=\int_Y\Big(\int_{\pi^{-1}(y)}f(x)\,d\mu_y(x) \Big )\,d\eta(y)
		\end{equation}
		for every Borel map $f:X \rightarrow [0,+\infty]$.
	\end{theorem}
\section{Constrained MFG equilibria} 
\subsection{Aproximation of constrained trajectories}
Let $\Omega$ be a bounded open subset of $\mathbb{R}^n$ with $\mathcal{C}^2$ boundary. Let $\Gamma$ be the metric subspace of $AC(0,T;\mathbb{R}^n)$ defined by
\begin{equation*}
\Gamma=\Big\{\gamma\in AC(0,T;\mathbb{R}^n): \gamma(t)\in\overline{\Omega}, \ \ \ \forall t\in[0,T] \Big\}.
\end{equation*}
For any $x\in\overline{\Omega}$, we set 
\begin{equation}
\Gamma[x]=\left\{ \gamma\in\Gamma :\gamma(0)=x\right\}.
\end{equation}
\begin{lemma}\label{lemmad}
	Let $\gamma\in AC(0,T;\mathbb{R}^n)$ and suppose that $d_\Omega(\gamma(t))<\rho_0$ for all $t \in[0,T]$. Then $d_\Omega \circ\gamma\in AC(0,T)$ and
	\begin{equation}\label{fderivata}
	\frac{d}{dt}(d_\Omega\circ \gamma)(t)=\big\langle Db_\Omega(\gamma(t)),\dot{\gamma}(t)\big\rangle\mathbf{1}_{\Omega^c}(\gamma(t)) \ \ \text{a.e.}\ t\in[0,T].
	\end{equation}
	Moreover, 
	\begin{equation}\label{ng}
	N_\gamma:=\left\{t\in[0,T]:\gamma(t)\in\partial\Omega,\ \exists\ \dot{\gamma}(t),\ \langle Db_\Omega(\gamma(t)),\dot{\gamma}(t)\rangle\neq 0 \right\}
	\end{equation}
	is a discrete set.
\end{lemma}
\begin{proof}
	First we prove that $N_\gamma$ is a discrete set. Let $t\in N_\gamma$, then there exists $\epsilon>0$ such that $\gamma(s)\notin \partial \Omega$ for any $s\in\left(]t-\epsilon,t+\epsilon[\ \setminus \{t\}\right)\cap[0,T]$. Therefore, $N_\gamma$ is composed of isolated points and so it is a discrete set.\\
	Let us now set $\phi(t)=(d_\Omega\circ \gamma)(t)$ for all $t\in[0,T]$.
	We note that $\phi \in AC(0,T)$ because it is the composition of $\gamma\in AC(0,T;\mathbb{R}^n)$ with the Lipschitz continuous function $d_\Omega(\cdot)$. Denote by $D$ the set of $t\in[0,T]$ such that there exists the first order derivative of $\gamma$ in $t$, i.e.,
	\begin{equation*}
	D=\left\{t\in[0,T]: \ \exists \ \dot{\gamma}(t)\ \right\}.
	\end{equation*}
	We observe that $D$ has full Lebesgue measure and we decompose $D$ in the following way:
	\begin{equation*}
	D=\underbrace{\left\{ t\in D :\gamma(t)\notin \partial\Omega\right\}}_{D_0}\cup \underbrace{\left\{ t\in D :\gamma(t)\in \partial\Omega\right\}}_{D_1}.
	\end{equation*}
	By \cite[Theorem 4, p. 129]{evans}, for all $t\in D_0$ the first order derivative of $\phi$ is equal to
	\begin{align*}
	\dot{\phi}(t)=
	\begin{cases}
	0 \ \ \ \ \ &\gamma(t)\in\Omega\\
	\big\langle Db_\Omega(\gamma(t)),\dot{\gamma}(t)\big\rangle &\gamma(t)\in \mathbb{R}^n\setminus\Omega.
	\end{cases}
	\end{align*} 
	Now, consider $t\in D_1\setminus N_\gamma$. 
	Since $\gamma(t)\in\partial\Omega$, one has that
	\begin{equation*}
	\frac{\phi(t+h)-\phi(t)}{h}=\frac{d_\Omega(\gamma(t+h))}{h},
	\end{equation*}
	for all $h>0$.
	Since $\gamma(t+h)=\gamma(t)+h\dot{\gamma}(t)+o(h)$ and $d_\Omega$ is Lipschitz continuous, we obtain
	\begin{eqnarray*}
		0\leq \frac{d_\Omega(\gamma(t+h))}{h}  \leq \frac{o(h)}{h}+\frac{d_\Omega(\gamma(t)+h\dot{\gamma}(t))}{h}.
	\end{eqnarray*}
	Hence, one has that
	\begin{equation}\label{limd}
	0\leq\liminf_{h\rightarrow 0}\frac{d_\Omega(\gamma(t+h))}{h}\leq\limsup_{h\rightarrow 0}\frac{d_\Omega(\gamma(t+h))}{h}\leq\limsup_{h\rightarrow 0}\frac{d_\Omega(\gamma(t)+h\dot{\gamma}(t))}{h}.
	\end{equation}
Moreover, by the regularity of $b_\Omega$, we obtain
	\begin{align}\label{dis}
	d_\Omega(\gamma(t)+h\dot{\gamma}(t))\leq |b_\Omega(\gamma(t)+h\dot{\gamma}(t))| \leq |h|\ |\langle Db_\Omega(\gamma(t)),\dot{\gamma}(t)\rangle| +o(h).
	\end{align}
	Thus, since $t\in D\setminus N_\gamma$, we conclude that
	\begin{equation}
	\limsup_{h\rightarrow 0}\frac{d_\Omega(\gamma(t)+h\dot{\gamma}(t))}{h}\leq\left |Db_\Omega(\gamma(t)),\dot{\gamma}(t)\rangle\right|=0.
	\end{equation}
	So $\dot{\phi}(t)=0$ and the proof is complete.
\end{proof}
\begin{proposition}\label{aprt}
	Let $x_i\in\overline{\Omega}$ be such that $x_i\rightarrow x$ and let $\gamma \in \Gamma[x]$. Then there exists ${\gamma}_i \in \Gamma[x_i]$ such that:
	\begin{enumerate}
		\item[(i)] ${\gamma}_i\rightarrow \gamma$ uniformly on $[0,T]$;
		\item[(ii)] $\dot{\gamma}_i\rightarrow \dot{\gamma}$ a.e. on $[0,T]$;
		\item[(iii)] $|\dot{\gamma}_i(t)|\leq C|\dot{\gamma}(t)|$ for any $i\geq1$, a.e. $t\in[0,T]$, and some constant $C\geq 0$.
	\end{enumerate}
\end{proposition}
\begin{proof}
	Let $\hgi$ be the trajectory defined by
	\begin{equation}
	\hgi(t)=\gamma(t)+x_i-x.
	\end{equation}
	We observe that $d_\Omega(\hgi(t))\leq \rho_0$ for all $t \in[0,T]$ and all sufficiently large $i$, say $i\geq i_0$. Indeed,
	 $$
	 d_\Omega(\hgi(t))\leq|\hgi(t)-\gamma(t)|=|x_i-x|.
	 $$
	 Since $x_i\rightarrow x$, we have that $d_\Omega(\hgi(t))\leq \rho_0$ for all $t \in[0,T]$ and $i\geq i_0$. We denote by $\gamma_i$ the projection of $\hgi$ on $\overline{\Omega}$, i.e.,
	\begin{equation}
	\gamma_i(t)=\hgi(t)-d_\Omega(\hgi(t))Db_\Omega(\hgi(t)) \ \ \ \forall t \in[0,T].
	\end{equation}
	We note that $\gamma_i\in \Gamma[x_i]$. Moreover, $\gamma_i$ converges uniformly to $\gamma$ on $[0,T]$. Indeed,
	\begin{equation*}
	|\gamma_i(t)-\gamma(t)|=|x_i-x -d_\Omega(\hgi(t))Db_\Omega(\hgi(t))|\leq 2|x_i-x|, \ \ \ \forall \ t\in[0,T].
	\end{equation*}
	By Lemma \ref{lemmad}, $d_\Omega(\hgi(\cdot)) \in AC(0,T)$ and $\frac{d}{dt}\left(d_\Omega(\hgi(t))\right)=\big\langle Db_\Omega(\hgi(t)),\dot{\widehat{\gamma}}_i(t)\big\rangle \mathbf{1}_{\Omega^c}(\hgi(t))$ a.e. $t\in[0,T]$. Using the regularity of $b_\Omega$, we obtain
	\begin{equation*}\label{deriv}
	\dot{\gamma}_i(t)=\dot{\gamma}(t)-\big\langle Db_\Omega(\hgi(t)),\dot{\gamma}(t)\big\rangle Db_\Omega(\hgi(t))\mathbf{1}_{\Omega^c}(\hgi(t)) - d_\Omega(\hgi(t)) D^2b_\Omega(\hgi(t))\dot{\gamma}(t),   \ \mbox{a.e.} \ t \in [0,T].
	\end{equation*}
	Therefore, there exists a constant $C\geq 0$ such that $|\dot{\gamma}_i(t)|\leq C|\dot{\gamma}(t)|$ for any $i\geq i_0$, a.e. $t\in[0,T]$.\\
Finally, we have to show that $\dot{\gamma}_i\rightarrow \dot{\gamma}$ almost everywhere on $[0,T]$. Since $\hgi\rightarrow \gamma$ and $\gamma\in\Gamma[x]$, one has that
	\begin{equation*}
	d_\Omega(\hgi(t))D^2b_\Omega(\hgi(t))\dot{\gamma}(t)\xrightarrow{i \rightarrow \infty} 0, \ \ \ \forall t\in [0,T].
	\end{equation*}
	So, we have to prove that 
	\begin{equation}\label{3.9}
	-\big\langle Db_\Omega(\hgi(t)),\dot{\gamma}(t)\big\rangle Db_\Omega(\hgi(t))\mathbf{1}_{\Omega^c}(\hgi(t))\xrightarrow{i\rightarrow \infty} 0, \ \mbox{a.e.}\ t \in[0,T].
	\end{equation}
	We note that
	\begin{equation}\label{k12}
	\Big|\big\langle Db_\Omega(\hgi(t)),\dot{\gamma}(t)\big\rangle Db_\Omega(\hgi(t))\mathbf{1}_{\Omega^c}(\hgi(t))\Big|\leq \Big|\big\langle Db_\Omega(\hgi(t)),\dot{\gamma}(t)\big\rangle\Big|, \ \ \mbox{a.e.} \ t \in[0,T].
	\end{equation}
	Fix $t\in[0,T]$ such that \eqref{k12} holds. If $\gamma(t)\in\Omega$ then $\hgi(t)\in \Omega$ for $i$ large enough and \eqref{3.9} holds. On the other hand, if $\gamma(t)\in\partial\Omega$, then passing to the limit in \eqref{k12}, we have that
	\begin{equation*}
	\limsup_{i\rightarrow \infty}\Big|\big\langle Db_\Omega(\hgi(t)),\dot{\gamma}(t)\big\rangle Db_\Omega(\hgi(t))\mathbf{1}_{\Omega^c}(\hgi(t))\Big|\leq\limsup_{i\rightarrow \infty} \Big|\big\langle Db_\Omega(\hgi(t)),\dot{\gamma}(t)\big\rangle\Big|.
	\end{equation*}
	Since $\gamma_i \rightarrow \gamma$ uniformly on $[0,T]$, one has that
	\begin{equation}
	\limsup_{i\rightarrow \infty} \Big|\big\langle Db_\Omega(\hgi(t)),\dot{\gamma}(t)\big\rangle\Big|=\Big|\big\langle Db_\Omega(\gamma(t)),\dot{\gamma}(t)\big\rangle\Big|.
	\end{equation}
	By Lemma \ref{lemmad}, we have that $\langle Db_\Omega(\gamma(t)),\dot{\gamma}(t)\rangle=0$ for $t\in [0,T]\setminus N_\gamma$, where $N_\gamma$ is the discret set defined in \eqref{ng}. So \eqref{3.9} holds for a.e. $t\in[0,T]$. Thus, $\dot{\gamma}_i$ converges almost everywhere to $\dot{\gamma}$ on $[0,T]$. This completes the proof. 
\end{proof}
\subsection{Assumptions}\label{ipotesi}
Let $\Omega$ be a bounded open subset of $\mathbb{R}^n$ with $C^2$ boundary. Let $\mathcal{P}(\overline{\Omega})$  be the set of all Borel probability measures on $\overline\Omega$ endowed with the Kantorovich-Rubinstein distance $d_1$ defined in \eqref{dis1}. We suppose throughout that $F,G:\overline{\Omega}\times\mathcal{P}(\overline{\Omega})\rightarrow \mathbb{R}$ and $L:\overline \Omega\times \mathbb{R}^n\to \mathbb{R}$ are given continuous functions.
Moreover, we assume the following conditions.
\begin{enumerate}
	\item[(L1)] $L\in C^1(\overline{\Omega}\times\mathbb{R}^n)$ and for all $(x,v)\in\overline \Omega\times \mathbb{R}^n$, 
	\begin{align}
	&|D_xL(x,v)|\leq C(1+|v|^2)\label{l11},\\
	&|D_vL(x,v)|\leq C(1+|v|)\label{l12},
	\end{align}
	for some constant $C>0$.
	\item[(L2)] There exist constants  $c_1,c_0>0$ such that
	\begin{equation}\label{l3b}
L(x,v)\geq c_1|v|^2-c_0 ,\ \ \ \ \ \forall (x,v)\in \overline{\Omega}\times \mathbb{R}^n.
	\end{equation} 
\item[(L3)] $v\longmapsto L(x,v)$ is convex for all $x\in\overline{\Omega}$.
\end{enumerate}
\begin{remark}\label{1o11}
$(i)$ As $\overline{\Omega}\times\mathcal{P}(\overline{\Omega})$ is a compact set, the continuity of $F$ and $G$  implies that they are bounded and uniformly continuous on $\overline{\Omega}\times\mathcal{P}(\overline{\Omega})$.\\ 
$(ii)$ In (L1), $L$ is assumed to be of class $C^1(\overline{\Omega}\times \mathbb{R}^n)$ just for simplicity. All the results of this paper hold true if $L$ is locally Lipschitz---hence, a.e. differentiable---in $\overline{\Omega}\times\mathbb{R}^n$ and satisfies the growth conditions \eqref{l11} and \eqref{l12} a.e. on $\overline{\Omega}\times \mathbb{R}^n$, see Remark \ref{red} below.
\end{remark}
\subsection{Existence of constrained MFG equilibria}
For any $t\in [0,T]$, we denote by $e_t:\Gamma\to \overline \Omega$ the evaluation map defined by 
$$
e_t(\gamma)= \gamma(t), \ \ \  \ \forall \gamma\in\Gamma.
$$
For any $\eta\in \mathcal{P}(\Gamma)$, we define
\begin{equation}\label{me}
m^\eta(t)=e_t\sharp \eta, \ \ \ \ \ \forall t \in [0,T].
\end{equation}
\begin{lemma}\label{3.2}
The following holds true.
	\begin{enumerate}
	\item[(i)] $m^\eta \in C([0,T];\mathcal{P}(\overline{\Omega}))$ for any $\eta\in\mathcal{P}(\Gamma)$.
	\item[(ii)]	Let $\eta_i$, $\eta\in\mathcal{P}(\Gamma)$, $i\geq 1$, be such that $\eta_i$ is narrowly convergent to $\eta$. Then $m^{\eta_i}(t)$ is narrowly convergent to $m^\eta(t)$  for all $t\in[0,T]$.
	\end{enumerate}
\end{lemma}
\begin{proof}
	First, we prove point $(i)$. By definition \eqref{me}, it is obvious that $m^\eta(t)$ is a Borel probability measure on $\overline{\Omega}$ for any $t\in[0,T]$. Let $\{t_k\}\subset[0,T]$ be a sequence such that $t_k\rightarrow \overline{t}$. We want to show that 
	\begin{equation}\label{m1}
	\lim_{t_k\rightarrow \overline{t}} \int_{\overline{\Omega}} f(x) m^\eta(t_k,\,dx)=\int _{\overline{\Omega}} f(x) m^\eta(\overline{t},\,dx),
	\end{equation}
	for any $f\in C(\overline{\Omega})$. Since $m^\eta(t_k)=e_{t_k}\sharp \eta$ and $e_{t_k}(\gamma)=\gamma(t_k)$, we have that
	\begin{equation*}
	\lim_{t_k\rightarrow \overline{t}} \int_{\overline{\Omega}} f(x) m^\eta(t_k,\,dx)=\lim_{t_k\rightarrow \overline{t}}\int_{\Gamma} f(e_{t_k}(\gamma))\,d\eta(\gamma)=\lim_{t_k\rightarrow \overline{t}}\int_{\Gamma} f(\gamma(t_k))\,d\eta(\gamma).
	\end{equation*}
	Since $f\in C(\overline{\Omega})$ and $\gamma\in \Gamma$, then $f(\gamma(t_k))\rightarrow f(\gamma(\overline{t}))$ and $|f(\gamma(t_k))|\leq ||f||_\infty$. Therefore, by Lebesgue's dominated convergence theorem, we have that
	\begin{equation}
	\lim_{t_k\rightarrow \overline{t}}\int_{\Gamma} f(\gamma(t_k))\,d\eta(\gamma)=\int_{\Gamma} f(\gamma(\overline{t}))\,d\eta(\gamma).
	\end{equation}
	Thus, recalling the definition of $m^\eta$, we obtain \eqref{m1}. Moreover, by Proposition \ref{cm}, we conclude that $d_1(m^\eta(t_k),m^\eta(\overline{t}))\rightarrow 0$. This completes the proof of point $(i)$.\\
	In order to prove point $(ii)$, we suppose that $\eta_i$ is narrowly convergent to $\eta$. Then, for all $f\in C(\overline{\Omega})$ we have that
	\begin{equation*}
	\lim_{i \rightarrow \infty} \int_{\overline \Omega} f(x)m^{\eta_i}(t,\,dx)=\lim_{i \rightarrow \infty} \int_{\Gamma} f(\gamma(t))\,d\eta_i(\gamma)=  \int_{\Gamma} f(\gamma(t)) \,d\eta(\gamma) = \int _{\overline \Omega} f(x)  m^\eta(t,\,dx).
	\end{equation*}
	Hence, $m^{\eta_i}(t)$ is narrowly convergent to $m^\eta(t)$ for all $t\in[0,T]$.
\end{proof}
\noindent
For any fixed $m_0\in\mathcal{P}(\overline{\Omega})$, we denote by ${\mathcal P}_{m_0}(\Gamma)$ the set of all Borel probability measures $\eta$ on $\Gamma$ such that $e_0\sharp \eta=m_0$.
For all $\eta \in \mathcal{P}_{m_0}(\Gamma)$, we define 
\begin{equation}
J_\eta [\gamma]=\int_0^T \Big[L(\gamma(t),\dot \gamma(t))+ F(\gamma(t),m^\eta(t))\Big]\ dt + G(\gamma(T),m^\eta(T)), \ \ \ \gamma\in\Gamma.
\end{equation}
\begin{remark}\label{nonempty}
	We note that ${\mathcal P}_{m_0}(\Gamma)$ is nonempty. Indeed, let $j:\overline{\Omega}\rightarrow \Gamma$ be the continuous map defined by
	$$
	j(x)(t)=x \ \ \ \ \forall t \in[0,T].
	$$
	Then,
	$$
	\eta :=j\sharp m_0
	$$
	is a Borel probability measure on $\Gamma$ and $\eta \in\mathcal{P}_{m_0}(\Gamma)$.
\end{remark}
\noindent
For all $x \in \overline{\Omega}$ and $\eta\in\mathcal{P}_{m_0}(\Gamma)$, we define
\begin{equation}
\Gamma^\eta[x]=\left\{ \gamma\in\Gamma[x]:J_\eta[\gamma]=\min_{\Gamma[x]} J_\eta\right\}.
\end{equation}
\begin{definition}
Let $m_0\in\mathcal{P}(\overline{\Omega})$. We say that $\eta\in\mathcal{P}_{m_0}(\Gamma)$ is a constrained MFG equilibrium for $m_0$ if
\begin{equation}
supp(\eta)\subseteq \bigcup_{x\in\overline{\Omega}} \Gamma^\eta[x].
\end{equation}
\end{definition}
\noindent
In other words, $\eta\in\mathcal{P}_{m_0}(\Gamma)$ is a constrained MFG equilibrium for $m_0$ if for $\eta$-a.e. $\overline{\gamma}\in \Gamma$ we have that
\begin{equation*}
 J_\eta[\overline{\gamma}]\leq J_\eta[\gamma], \ \ \ \ \ \ \forall \gamma \in\Gamma[\overline{\gamma}(0)].
\end{equation*}
The main result of this section is the existence of constrained MFG equilibria for $m_0$.
\begin{theorem}\label{tesistenza}
	Let $\Omega$ be a bounded open subset of $\mathbb{R}^n$ with $C^2$ boundary and let $m_0\in\mathcal{P}(\overline{\Omega})$. Suppose that (L1)-(L3) hold true. Let $F:\overline{\Omega}\times\mathcal{P}(\overline{\Omega})\rightarrow \mathbb{R}$ and $G:\overline{\Omega}\times \mathcal{P}(\overline{\Omega})\rightarrow \mathbb{R}$ be continuous. Then, there exists at least one constrained MFG equilibrium for $m_0$.
\end{theorem}
\noindent
Theorem \ref{tesistenza} will be proved in Section \ref{sub}. Now, we will show some properties of $\Gamma^\eta[x]$ that we will use in what follows.
\begin{lemma}\label{pr1}
For all $x\in\overline{\Omega}$ and $\eta\in\mathcal{P}_{m_0}(\Gamma)$ the following holds true.
\begin{enumerate}
\item[(i)] $\Gamma^\eta[x]$ is a nonempty set. 
\item[(ii)] All $\gamma\in\Gamma^\eta[x]$ satisfy
\begin{equation}\label{k1}
||\dot{\gamma}||_2\leq K,
\end{equation}
where
\begin{equation}\label{k11}
K=\frac{1}{\sqrt{c_1}}\Big[T\max_{\overline{\Omega}}L(x,0)+2T\max_{\overline{\Omega}\times \mathcal{P}(\overline{\Omega})}|F|+2\max_{\overline{\Omega}\times \mathcal{P}(\overline{\Omega})}|G|+Tc_0\Big]^{\frac{1}{2}}
\end{equation}
and $c_0$, $c_1$ are the constants in \eqref{l3b}. Consequently, all minimizers $\gamma \in \Gamma^\eta[x]$ are $\frac{1}{2}$-H\"{o}lder continuous of constant $K$.
\end{enumerate}
In addition, if $\eta\in \mathcal{P}_{m_0}(\Gamma)$ is a constrained MFG equilibrium for $m_0$, then $m^\eta(t)=e_t\sharp \eta$ is $\frac{1}{2}$-H\"{o}lder continuous of constant $K$.
\end{lemma}
\begin{proof}
By classical results in the calculus of variation (see, e.g., \cite[ Theorem 6.1.2]{c}), there exists at least one mimimizer of $J_\eta[\cdot]$ on $\Gamma$ for any fixed initial point $x\in \overline{\Omega}$. So $\Gamma^\eta[x]$ is a nonempty set.\\  
Let $x\in\overline{\Omega}$ and let $\gamma\in \Gamma^\eta[x]$. By comparing the cost of $\gamma$ with the cost of the constant trajectory $\gamma(0)\equiv x$, one has that%
\begin{align}\label{2}
&\int_0^T \Big[L(\gamma(t),\dot \gamma(t)) + F(\gamma(t),m^\eta(t))\Big] \, dt + G(\gamma(T),m^\eta(T))\\
&\leq \int _0^T \Big[L(x,0)+ F(x,m^\eta(t))\Big] \,dt + G(x,m^\eta(T)) \nonumber\\
&\leq \Big[T\max_{\overline{\Omega}}L(x,0)+T\max_{\overline{\Omega}\times \mathcal{P}(\overline{\Omega})}|F|+\max_{\overline{\Omega}\times \mathcal{P}(\overline{\Omega})}|G|\Big].\nonumber
\end{align}
Using \eqref{l3b} in (\ref{2}), one has that
\begin{equation}\label{4}
||\dot\gamma||_2 \leq \frac{1}{\sqrt{c_1}}\Big[T\max_{\overline{\Omega}}L(x,0)+2T\max_{\overline{\Omega}\times \mathcal{P}(\overline{\Omega})}|F|+2\max_{\overline{\Omega}\times \mathcal{P}(\overline{\Omega})}|G|+Tc_0\Big]^{\frac{1}{2}}=K,
\end{equation}
where $c_0$, $c_1$ are the constants in \eqref{l3b}. This completes the proof of point $(ii)$ since the H\"{o}lder regularity of $\gamma$ is a direct conseguence of the estimate \eqref{4}.\\
Finally, we claim that, if $\eta$ is a constrained MFG equilibrium for $m_0$, then the map $t \rightarrow m^{\eta}(t)$ is $\frac{1}{2}$-H\"older continuous with constant $K$. Indeed, for any $t_1,t_2\in[0,T]$, we have that
\begin{equation}
d_1(m^{\eta}(t_2),m^{\eta}(t_1))= \sup_{\phi} \int_{\bar \Omega} \phi(x)\,(m^{\eta}(t_2,dx)-m^{\eta}(t_1,dx)),
\end{equation}
where the supremum is taken over the set of all 1-Lipschitz continuous maps $\phi:\overline \Omega\rightarrow \mathbb{R}$. Since $m^{\eta}(t)=e_t\sharp \eta$ and the map $\phi$ is 1-Lipschitz continuous, one has that
\begin{align*}
&\int_{\overline\Omega} \phi(x)\,(m^{\eta}(t_2,dx)-m^{\eta}(t_1,dx))=\int_{\Gamma}\Big[ \phi(e_{t_2}(\gamma))-\phi(e_{t_1}(\gamma))\Big] \,d\eta(\gamma)\\
&=\int_{\Gamma} \Big[\phi(\gamma(t_2))-\phi(\gamma(t_1))\Big] \,d\eta(\gamma)\leq \int_{\Gamma} |\gamma(t_2)-\gamma(t_1)|\,d\eta(\gamma).
\end{align*}
Since $\eta$ is a constrained MFG equilibrium for $m_0$, property $(ii)$ yields
\begin{align*}
\int_{\Gamma} |\gamma(t_2)-\gamma(t_1)|\,d\eta(\gamma)\leq K\int_{\Gamma} |t_2-t_1|^{\frac{1}{2}}\,d\eta(\gamma)=K|t_2-t_1|^{\frac{1}{2}}.
\end{align*}
Hence, we conclude that
\begin{equation*}
d_1(m^{\eta}(t_2),m^{\eta}(t_1)) \leq K |t_2-t_1|^{\frac{1}{2}}, \ \ \ \ \forall\ t_1, t_2 \in [0,T]
\end{equation*}
and the map $t\longmapsto m^{\eta}(t)$ is 1/2-H\"older continuous.
 \end{proof}
 \begin{lemma}\label{3.4}
 Let $\eta_i$, $\eta\in \mathcal{P}_{m_0}(\Gamma)$ be such that $\eta_i$ narrowly converges to $\eta$. Let $x_i\in\overline{\Omega}$ be such that $x_i\rightarrow x$ and let $\gamma_i\in \Gamma^{\eta_i}[x_i]$ be such that $\gamma_i\rightarrow \bg$. Then $\bg\in\Gamma^\eta[x]$. Consequently, $\Gamma^\eta[\cdot]$ has closed graph.
 \end{lemma}
 \begin{proof}
 We want to prove that 
 \begin{equation}\label{mini1}
 J_\eta[{\bg}]\leq J_\eta[\gamma], \ \ \  \forall \gamma\in \Gamma[x]\ \mbox{such that} \ \int_0^T|\dot{\gamma}|^2\,dt <\infty.
 \end{equation}
We observe that the above request is not restrictive because, by assuption (L2), if $\int_0^T|\dot{\gamma}|^2\,dt=\infty$ then the above inequality is trivial.\\
Fix $\gamma\in\Gamma[x]$ with $\int_0^T|\dot{\gamma}|^2\,dt <\infty$, by Proposition \ref{aprt}, we have that there exists $\widehat{\gamma}_i\in \Gamma[x_i]$ such that $\widehat{\gamma}_i\rightarrow \gamma$ uniformly on $[0,T]$, $\dot{\widehat{\gamma}}_i\rightarrow \dot{\gamma}$ a.e. on $[0,T]$ and $|\dot{\widehat{\gamma}}_i(t)|\leq C|\dot{\gamma}(t)|$ for any $i\geq 1$, a.e. $t\in[0,T]$, and some constant $C\geq 0$. 
 Since $\gamma_i\in\Gamma^{\eta_i}[x_i]$, one has that
 \begin{equation}
 J_{\eta_i}[\gamma_i]\leq J_{\eta_i}[\widehat{\gamma}_i],\ \  \ \ \forall i\geq 1.
 \end{equation}
 So, in order to prove \eqref{mini1}, we have to check that
 \begin{enumerate}
 	\item[(a)] $J_\eta[\bg]\leq\liminf_{i\rightarrow \infty} J_{\eta_i}[{\gamma}_i]$;
 	\item[(b)] $\lim_{i\rightarrow +\infty} J_{\eta_i}[\widehat{\gamma}_i]=J_\eta[\gamma]$.
 \end{enumerate}
 First we show that $(a)$ holds, that is,
 \begin{align}\label{linf}
 &\liminf_{i\rightarrow \infty} \Big \{\int_0^T \Big[L\big(\gamma_{i}(t),\dot{\gamma}_{i}(t))\big)+F\big(\gamma_{i}(t),m^{\eta_i}(t)\big)\Big]\,dt+G(\gamma_{i}(T),m^{\eta_i}(T))\Big\}\nonumber\\
 &\geq \int_0^T[L\big(\bg(t),\dot{\bg}(t))\big)+F\big(\bg(t),m^{\eta}(t)\big)\Big]\,dt+G(\bg(T),m^{\eta}(T)).
 \end{align}
 First of all, we recall that, by Lemma \ref{3.2}, $m^{\eta_i}(t)$ narrowly converges to $m^{\eta}(t)$ for all $t \in [0,T]$.
 Owing to the convergence of $\gamma_{i}$ to $\bg$, the  narrow convergence of $m^{\eta_i}(t)$ to $m^\eta(t)$, and our assumption on $F$ and $G$, we conclude that
 \begin{align*}
 \int_0^T F(\gamma_{i}(t),m^{\eta_i}(t))\,dt &\xrightarrow{i\rightarrow \infty} \int_0^T F(\bg(t),m^{\eta}(t))\,dt, \\
 G(\gamma_{i}(T),m^{\eta_i}(T)) &\xrightarrow{i\rightarrow \infty} G(\bg(T),m^{\eta}(T)).
 \end{align*}
 Up to taking a subsequence of $\gamma_i$, we can assume that $\dot{\gamma}_{i}\rightharpoonup\dot{\bg}$ in $L^2(0,T;\mathbb{R}^n)$ without loss of generality. By assumption (L3), one has that  
 \begin{align*}
 &\int_0^T L(\gamma_{i}(t),\dot{\gamma}_{i}(t))=\int_0^T L(\bg(t),\dot{\gamma}_{i}(t))\,dt+\int_0^T \Big[ L(\gamma_{i}(t),\dot{\gamma}_{i}(t))-L(\bg(t),\dot{\gamma}_{i}(t))\Big]\,dt\\
 &\geq \int_0^T  \Big[L(\bg(t),\dot{\bg}(t))+ \langle D_vL(\bg(t),\dot{\bg}(t)),\dot{\gamma}_{i}-\dot{\bg} \rangle\Big]\,dt+\int_0^T \Big[ L(\gamma_{i}(t),\dot{\gamma}_{i}(t))-L(\bg(t),\dot{\gamma}_{i}(t))\Big]\,dt.
 \end{align*}
 Since $\gamma_{i}\in \Gamma^{\eta_i}[x_i]$ and $ \gamma_i\rightarrow \bg$, by (L1), we obtain
 \begin{align*}
 \int_0^T \Big[L(\gamma_i(t),\dot{\gamma}_i(t))-L(\bg(t),\dot{\gamma}_i(t))\Big]\,dt \xrightarrow{i\rightarrow \infty} 0.
 \end{align*}
 Moreover, since $\dot{\gamma}_i\rightharpoonup\dot{\bg}$ in $L^2(0,T;\mathbb{R}^n)$, one has that
 \begin{equation*}
 \int_0^T \big\langle D_vL(\bg(t),\dot{\bg}(t)),\dot{\gamma}_i-\dot{\bg} \big\rangle\,dt \xrightarrow{i\rightarrow \infty} 0.
 \end{equation*}
 Thus, \eqref{linf} holds.\\ Finally, we prove $(b)$, i.e.,
 \begin{align*}
& \lim_{i\rightarrow \infty}\left\{ \int_0^TL(\widehat{\gamma}_i(t),\dot{\widehat{\gamma}}_i(t))+F(\widehat{\gamma}_i(t),m^{\eta_i}(t))\,dt+G(\widehat{\gamma}_i(T),m^{\eta_i}(T))\right\}\\
 &=\int_0^TL(\gamma(t),\dot{\gamma}(t))+F(\gamma(t),m^\eta(t))\,dt+G(\gamma(T),m^\eta(T)).
 \end{align*}
 Owing to the convergence of $\widehat{\gamma}_i$ to $\gamma$, the narrow convergence of $m^{\eta_i}(t)$ to $m^\eta(t)$ for all $t\in[0,T]$, and our assumption on $F$ and $G$, one has
 \begin{align*}
 \int_0^T F(\widehat{\gamma}_i(t),m^{\eta_i}(t))\,dt &\xrightarrow{i\rightarrow \infty} \int_0^T F(\gamma(t),m^{\eta}(t))\,dt, \\
 G(\widehat{\gamma}_i(T),m^{\eta_i}(T)) &\xrightarrow{i\rightarrow \infty} G(\gamma(T),m^{\eta}(T)).
 \end{align*}
 Hence, we only need to prove that 
 \begin{equation}\label{linf3}
 \liminf_{i\rightarrow \infty} \int_0^T L(\widehat{\gamma}_i(t),\dot{\widehat{\gamma}}_i(t))\,dt=\int_0^T L(\gamma(t),\dot{\gamma}(t))\,dt.
 \end{equation}
 Owing to (L1), one has that
 \begin{align*}
 &\Big|\int_0^T \big[L(\widehat{\gamma}_i(t),\dot{\widehat{\gamma}}_i(t))-L(\gamma(t),\dot{\gamma}(t))\big]\,dt\Big|\\
 &\leq \int_0^T \Big|L(\widehat{\gamma}_i(t),\dot{\widehat{\gamma}}_i(t))-L(\gamma(t),\dot{\widehat{\gamma}}_i(t))\Big|\,dt+\int_0^T \Big|L(\gamma(t),\dot{\widehat{\gamma}}_i(t))-L(\gamma(t),\dot{\gamma}(t))\Big|\,dt\\
 &\leq ||\widehat{\gamma}_i-\gamma||_\infty\int_0^T \Big(1+|\dot{\widehat{\gamma}}_i|^2\Big)\,dt+\int_0^T \Big|\int_0^1\Big\langle D_vL\big(\gamma(t),\lambda\dot{\widehat{\gamma}}_i+(1-\lambda)\dot{\gamma}\big),\dot{\widehat{\gamma}}_i(t)-\dot{\gamma}(t)\Big\rangle\,d\lambda\Big|\,dt\\
 &\leq ||\widehat{\gamma}_i-\gamma||_\infty\int_0^T \Big(1+|\dot{\widehat{\gamma}}_i|^2\Big)\,dt + C\int_0^T\int_0^1 \Big(1+|\dot{\widehat{\gamma}}_i|+|\dot{\gamma}|\Big)|\dot{\widehat{\gamma}}_i(t)-\dot{\gamma}(t)|\,dt.
 \end{align*}
 Since $\widehat{\gamma}_i\rightarrow \gamma$ uniformly on $[0,T]$ and $|\dot{\widehat{\gamma}}_i(t)|\leq C|\dot{\gamma}(t)|$ for any $i\geq 1$ and for any $t\in[0,T]$, we have that
 \begin{equation*}
 ||\widehat{\gamma}_i-\gamma||_\infty\int_0^T \Big(1+|\dot{\widehat{\gamma}}_i|^2\Big)\,dt\xrightarrow{i\rightarrow \infty} 0.
 \end{equation*}
 In addition, since $\dot{\widehat{\gamma}}_i\rightarrow \dot{\gamma}$ a.e. on $[0,T]$, by Lebesgue's dominated convergence theorem we obtain
 \begin{equation}
 C\int_0^T\int_0^1 \Big(1+|\dot{\widehat{\gamma}}_i|+|\dot{\gamma}|\Big)\left|\dot{\widehat{\gamma}}_i(t)-\dot{\gamma}(t)\right|\,dt\xrightarrow{i\rightarrow \infty} 0.
 \end{equation}
 This gives $(b)$ and the proof is complete.
 \end{proof}
 \begin{remark}\label{red}
 The above proof can be adapted to treat the case of a locally Lipschitz Lagrangian $L$ as was mentioned in Remark \ref{1o11}. Indeed, it suffices to replace the gradient $D_vL(\overline{\gamma}(t),\dot{\overline{\gamma}}(t))$ with a measurable selection of the subdifferential $\partial_vL(\overline{\gamma}(t),\dot{\overline{\gamma}}(t))$.
 \end{remark}
\subsection{Proof of Theorem \ref{tesistenza}}\label{sub}
In this section we prove Theorem \ref{tesistenza} using a fixed point argument. First of all, we recall that, by Theorem \ref{td}, for any $\eta\in\mathcal{P}_{m_0}(\Gamma)$, there exists a unique Borel measurable family of probabilities $\{\eta_x\}_{x\in\overline{\Omega}}$ on $\Gamma$ which disintegrates $\eta$ in the sense that
\begin{equation}\label{dise}
\begin{cases}
	\eta(d\gamma)=\int_{\overline{\Omega}} \eta_x(d\gamma) \,dm_0(x),\\
	supp(\eta_x)\subset \Gamma[x] \ \ m_0-\mbox{a.e.} \ x\in \overline{\Omega}.
	\end{cases}
	\end{equation}
We introduce the set-valued map $
E:\mathcal{P}_{m_0}(\Gamma) \rightrightarrows \mathcal{P}_{m_0}(\Gamma)
$
by defining, for any $\eta\in\mathcal{P}_{m_0}(\Gamma)$, 
\begin{equation}\label{eeta}
E(\eta)=\Big\{\widehat{\eta}\in\mathcal{P}_{m_0}(\Gamma): supp(\widehat{\eta}_x)\subseteq \Gamma^\eta[x] \ m_0-\text{a.e.}\ x\in\overline{\Omega}\Big\}.
\end{equation}
Then, it is immediate to realize that $\eta\in\mathcal{P}_{m_0}(\Gamma)$ is a constrained MFG equilibrium for $m_0$ if and only if $\eta\in E(\eta)$.
We will therefore show that the set-valued map $E$ has a fixed point. For this purpose, we will apply Kakutani's Theorem \cite{ka}. The following lemmas are intended to check that the assumptions of such a theorem are satisfied by $E$.
\begin{lemma}\label{l} 
For any $\eta\in \mathcal{P}_{m_0}(\Gamma)$, $E(\eta)$ is a nonempty convex set.
\end{lemma}
\begin{proof}
First, we note that $E(\eta)$ is a nonempty set. Indeed, by $(i)$ of Lemma \ref{pr1}, Lemma \ref{3.4}, and \cite[ Theorem 8.1.4]{af} we have that $x\longmapsto \Gamma^\eta[x]$ is measurable. Moreover, by \cite[ Theorem 8.1.3]{af}, $x\longmapsto \Gamma^\eta[x]$ has a Borel measurable selection $x\longmapsto{\gamma}^\eta_x$. Thus, the measure $\widehat{\eta}$, defined by $\widehat{\eta}(\mathcal{B})=\int_{\overline{\Omega}} \delta_{\{{\gamma}^\eta_x\}}(\mathcal{B})\,dm_0(x)$ for any $\mathcal{B}\in\mathscr{B}(\Gamma)$, belongs to $E(\eta)$. \\
Now we want to check that $E(\eta)$ is a convex set.
Let $\{\eta_i\}_{i=1,2} \in E(\eta)$ and let $\lambda_1 , \lambda_2 \geq 0$ be such that $\lambda_1+\lambda_2=1$. Since $\eta_i$ are Borel probability measures, $\widehat{\eta}:= \lambda_1\eta_1 + \lambda_2 \eta_2$ is a Borel probability measure as well. Moreover, for any Borel set $B\in \mathscr{B}(\overline{\Omega})$ we have that
\begin{equation}
e_0 \sharp \widehat{\eta} (B)=\widehat{\eta} (e_0^{-1}(B))= \sum_{i=1}^{2} \lambda_i \eta_i(e_0^{-1}(B))=\sum_{i=1}^{2} \lambda_i e_0 \sharp \eta_i(B)=\sum_{i=1}^{2} \lambda_i m_0(B) =m_0 (B).
\end{equation}
So, $\widehat{\eta} \in {\mathcal P}_{m_0}(\Gamma)$. Since $\eta_1 \in E(\eta)$, there exists a unique Borel measurable family of probabilities $\{\eta_{1,x}\}_{x\in\overline{\Omega}}$ on $\Gamma$ which disintegrates $\eta_1$ as in \eqref{dise} and there exists $A_1\subset \overline{\Omega}$, with $m_0(A_1)=0$, such that
\begin{equation}
supp(\eta_{1,x})\subset \Gamma^\eta[x], \ \ \ x\in \overline{\Omega}\setminus A_1.
\end{equation}
In the same way, $\eta_2(d\gamma)=\int_{\overline{\Omega}}\eta_{2,x}(\,d\gamma)\,dm_0(x)$ can be disintegrated and one has that 
\begin{equation}
supp(\eta_{2,x})\subset \Gamma^\eta[x] \ \ \ x\in \overline{\Omega}\setminus A_2,
\end{equation}
where $A_2$ is such that $m_0(A_2)=0$.
Hence, $\widehat{\eta}$ can be disintegrated in the following way:  
\begin{equation}
\begin{cases}
\widehat{\eta}(d\gamma)= \int_{\overline \Omega} \Big(\lambda_1\eta_{1,x}+\lambda_2 \eta_{2,x}\Big)(d\gamma) dm_0(x),\\
supp(\lambda_1\eta_{1,x}+\lambda_2 \eta_{2,x})\subset \Gamma^\eta[x] \ \ \ x\in \overline{\Omega}\setminus(A_1\cup A_2),
\end{cases}
\end{equation}
where $m_0(A_1\cup A_2)=0$. This completes the proof.
\end{proof}
\noindent
	\begin{lemma}\label{gchiuso}
	The map $E:\mathcal{P}_{m_0}({\Gamma}) \rightrightarrows \mathcal{P}_{m_0}({\Gamma})$ has closed graph. 
	\end{lemma}
	\begin{proof}
	Let $\eta_i$, $\eta$ $\in {\mathcal P}_{m_0}({\Gamma})$ be such that $ \eta_i$ is narrowly convergent to $\eta$. Let $\widehat \eta_i \in E(\eta_i)$ be such that $ \widehat\eta_i$ is narrowly convergent to $\widehat \eta$. We have to prove that $\widehat \eta \in E(\eta)$. 
Since $\widehat{\eta}_i$ narrowly converges to $\widehat \eta$, we have that $\widehat \eta\in \mathcal{P}_{m_0}(\Gamma)$ and there  exists a unique Borel measurable family of probabilities $\{\widehat{\eta}_{x}\}_{x\in\overline\Omega}$ on $\Gamma$ such that $\widehat{\eta}(d\gamma)=\int_{\overline{\Omega}} \widehat{\eta}_{x}(d\gamma)\,dm_0(x)$ and $supp(\widehat{\eta}_x)\subset \Gamma[x]$ for $m_0$-a.e. $x \in \overline{\Omega}$. Hence, $\widehat{\eta}\in E(\eta)$ if and only if $supp(\widehat\eta_x)\subseteq \Gamma^\eta[x]$ for $m_0$-a.e $x\in\overline{\Omega}$. Let $\Omega_0\subset \overline{\Omega}$ be an $m_0$-null set such that
\begin{equation*}
 supp(\widehat{\eta}_x)\subset \Gamma[x] \ \ \ \ \forall x\in\overline{\Omega}\setminus \Omega_0.
 \end{equation*}
 Let $x\in\overline{\Omega}\setminus \Omega_0$ and let $\widehat{\gamma}\in supp(\widehat{\eta}_x)$.
Since $\widehat{\eta}_i$ narrowly converges to $\widehat{\eta}$, then, by Proposition \ref{p21}, one has that  
\begin{equation*}
\exists \ \widehat\gamma_{i} \in \ supp (\widehat \eta_{i}) \ \ \mbox{such that} \  \lim_{i\rightarrow \infty} \widehat\gamma_{i} =\widehat\gamma.
\end{equation*} 
Let $\widehat{\gamma}_i(0)=x_i$, with $x_i\in\overline{\Omega}$.
Since $\widehat{\eta}_i\in E(\eta_i)$ and $\widehat{\gamma}_{i} \in supp(\widehat{\eta}_{i})$, we have that $\widehat{\gamma}_{i}$ is an optimal trajectory for $J_{\eta_i}[\cdot]$, i.e.,
\begin{equation}\label{mini}
J_{\eta_i}[\widehat{\gamma}_{i}] \leq J_{\eta_i}[\gamma] \ \ \ \ \forall \gamma\in\Gamma[x_i].
\end{equation}
As $\eta_i$ narrowly converges to $\eta$, applying Lemma \ref{3.4}, we conclude that $\widehat{\gamma}\in\Gamma^\eta[x]$. Since $x$ is any point in $\overline{\Omega}\setminus\Omega_0$, we have shown that $\widehat{\eta}\in E(\eta)$.
\end{proof}
\noindent
We denote by $\Gamma_K$ the set of trajectories $\gamma\in\Gamma$ such that $\gamma$ satisfies \eqref{k1}, i.e.,
\begin{equation}\label{gk}
\Gamma_K=\left\{\gamma\in\Gamma :||\dot\gamma||_2\leq K \right\}
\end{equation}
where $K$ is the constant given by \eqref{k11}. By the definition of $E(\eta)$ in \eqref{eeta} and Lemma \ref{pr1}, we deduce that
\begin{equation}
E(\eta)\subseteq \mathcal{P}_{m_0}(\Gamma_K) \ \ \ \forall \eta\in \mathcal{P}_{m_0}(\Gamma).
\end{equation} 
\begin{remark}\label{r1}
In general $\Gamma$ fails to be complete as a metric space. Then, by Theorem \ref{tp}, $\mathcal{P}_{m_0}(\Gamma)$ is not a compact set. On the other hand, if $\Gamma$ is replaced by $\Gamma_K$ then $\mathcal{P}_{m_0}(\Gamma_K)$  is a compact convex subset of $\mathcal{P}_{m_0}(\Gamma)$. Indeed, the convexity of $\mathcal{P}_{m_0}(\Gamma_K)$ follows by the same argument used in the proof of Lemma \ref{l}. As for compactness, let $\{\eta_k\}\subset \mathcal{P}_{m_0}(\Gamma_K)$. Since $\Gamma_K$ is a compact set, $\{\eta_k\}$ is tight. So, by Theorem \ref{p21}, one finds a subsequence, that we denote again by $\eta_k$, which narrowly converges to some probability measure $\eta\in\mathcal{P}_{m_0}(\Gamma_K)$.
\end{remark}
\noindent
We will restrict domain of interest to $\mathcal{P}_{m_0}(\Gamma_K)$, where ${\Gamma}_K$ is given by \eqref{gk}. Hereafter, we denote by $E$ the restriction $E_{\mid_{\mathcal{P}_{m_0}(\Gamma_K)}}$.\\
\underline{\textit{Conclusion}}\\
By Remark \ref{r1} and Remark \ref{nonempty}, ${\mathcal P}_{m_0}(\Gamma_K)$ is a nonempty compact convex set. Moreover, by Lemma \ref{l}, $E(\eta)$ is nonempty convex set for any $\eta\in\ {\mathcal P}_{m_0}(\Gamma_K)$ and, by Lemma \ref{gchiuso}, the set-valued map $E$ has closed graph. Then, the assumptions of Kakutani's Theorem are satisfied and so there exists  $\overline \eta\in {\mathcal P}_{m_0}(\Gamma_K)$ such that  $\overline \eta\in E(\overline \eta)$. 
\section{Mild solution of the constrained MFG problem}
In this section we define mild solutions of the constrained MFG problem in $\overline{\Omega}$. Moreover, under the assumptions of Section \ref{ipotesi}, we prove the existence of such solutions. Then, we give a uniqueness result under a classical monotonicity assumption on $F$ and $G$.
\begin{definition}
	We say that $(u,m)\in  C([0,T]\times \overline{\Omega})\times C([0,T];\mathcal{P}(\overline{\Omega}))$ is a mild solution of the constrained MFG problem in $\overline{\Omega}$ if there exists a constrained MFG equilibrium $\eta\in\mathcal{P}_{m_0}(\Gamma)$ such that 
	\begin{enumerate}
	\item [(i)] $ m(t)= e_t\sharp \eta$ for all $t\in[0,T]$;
	\item[(ii)] $u$ is given by
	\begin{equation}\label{v}
	u(t,x)=  \inf_{\tiny\begin{array}{c}
		\gamma\in \Gamma\\
		\gamma(t)=x
		\end{array}} 
	\left\{\int_t^T \left[L(\gamma(s),\dot \gamma(s))+ F(\gamma(s), m(s))\right]\ ds + G(\gamma(T),m(T))\right\},  
	\end{equation} 
	for $(t,x)\in [0,T]\times \overline{\Omega}$.
	\end{enumerate}
\end{definition}
\noindent
A direct consequence of Theorem \ref{tesistenza} is the following result.
\begin{corollary}\label{cesistenza}
Let $\Omega$ be a bounded open subset of $\mathbb{R}^n$ with $C^2$ boundary and let $m_0\in\mathcal{P}(\overline{\Omega})$. Suppose that (L1)-(L3) hold true. Let $F:\overline{\Omega}\times\mathcal{P}(\overline{\Omega})\rightarrow \mathbb{R}$ and $G:\overline{\Omega}\times \mathcal{P}(\overline{\Omega})\rightarrow \mathbb{R}$ be continuous. There exists at least one mild solution $(u,m)$ of the constrained MFG problem in $\overline{\Omega}$.
\end{corollary}
\noindent
Before proving our uniqueness result, we recall the following definitions.
\begin{definition}
	We say that $F:\overline{\Omega}\times \mathcal{P}(\overline{\Omega})\rightarrow \mathbb{R}$ is monotone if
	\begin{equation}\label{g}
	\int_{\overline \Omega} (F(x,m_1)-F(x,m_2))d(m_1-m_2)(x) \geq 0,
	\end{equation}
	for any $m_1, m_2 \in{\mathcal P}(\overline \Omega)$.\\
	We say that $F$ is strictly monotone if 
	\begin{equation}\label{f1}
	\int_{\overline \Omega} (F(x,m_1)-F(x,m_2))d(m_1-m_2)(x)\ \geq\ 0,
	\end{equation}
	for any $m_1,m_2\in {\mathcal P}(\overline \Omega)$ and $\int_{\overline \Omega} (F(x,m_1)-F(x,m_2))d(m_1-m_2)(x)=0$ if and only if $F(x,m_1)=F(x,m_2)$ for all $x\in\overline{\Omega}$.
	\end{definition}
\begin{example}
Assume that $F:\overline{\Omega}\times\mathcal{P}(\overline{\Omega})\rightarrow \mathbb{R}$ is of the form
\begin{equation}
F(x,m)=\int_{\overline{\Omega}} f(y,(\phi\star m)(y))\phi(x-y)\,dy,
\end{equation}
where $\phi:\mathbb{R}^n\rightarrow \mathbb{R}$ is a smooth even kernel with compact support and $f:\overline{\Omega}\times\mathbb{R}\rightarrow\mathbb{R}$ is a smooth function such that $z\rightarrow f(x,z)$ is strictly increasing for all $x\in\overline{\Omega}$. Then $F$ satisfies condition \eqref{f1}.\\
Indeed, for any $m_1$, $m_2\in\mathcal{P}(\overline{\Omega})$, we have that
\begin{align*}
&\int_{\overline{\Omega}}\left[F(x,m_1)-F(x,m_2)\right]\,d\left(m_1-m_2\right)(x)\\
&=\int_{\overline{\Omega}}\int_{\overline{\Omega}} \left[f(y,(\phi\star m_1)(y))-f(y,(\phi\star m_2)(y))\right]\phi(x-y)\,dy\,d\left(m_1-m_2\right)(x)\\
&=\int_{\overline{\Omega}} \left[f(y,(\phi\star m_1)(y))-f(y,(\phi\star m_2)(y))\right]\int_{\overline{\Omega}}\phi(x-y)\,d\left(m_1-m_2\right)(x)\,dy\\
&=\int_{\overline{\Omega}}\left[f(y,(\phi\star m_1)(y))-f(y,(\phi\star m_2)(y))\right]\left[(\phi\star m_1)(y)-(\phi\star m_2)(y)\right]\,dy.
\end{align*}
Since $z\rightarrow f(x,z)$ is increasing, then one has that
\begin{equation*}
\int_{\overline{\Omega}}\left[f(y,(\phi\star m_1)(y))-f(y,(\phi\star m_2)(y))\right]\left[(\phi\star m_1)(y)-(\phi\star m_2)(y)\right]\,dy\geq 0.
\end{equation*}
Moreover, if the term on the left-side above is equal to zero, then we obtain
\begin{equation*}
\left[f(y,(\phi\star m_1)(y))-f(y,(\phi\star m_2)(y))\right]\left[(\phi\star m_1)(y)-(\phi\star m_2(y))\right]=0 \ \ \ \mbox{a.e.}\ \  y\in\overline{\Omega}.
\end{equation*}
As $z\rightarrow f(x,z)$ is strictly increasing, we deduce that $\phi\star m_1(y)=\phi\star m_2(y)$ for any $y\in \overline{\Omega}$ and so $F(\cdot,m_1)=F(\cdot,m_2)$.
\end{example}
\begin{theorem}\label{u}
	Suppose that $F$ and $G$ satisfy \eqref{f1}. Let $\eta_1$, $\eta_2\in \mathcal{P}_{m_0}(\Gamma)$ be constrained MFG equilibria and let $J_{\eta_1}$ and $J_{\eta_2}$ be the associated functionals.
	Then $J_{\eta_1}$ is equal to $J_{\eta_2}$. Consequently, if $(u_1,m_1)$, $(u_2,m_2)$ are mild solutions of the constrained MFG problem in $\overline{\Omega}$, then $u_1=u_2$.
\end{theorem}
\begin{proof}
	Let $\eta_1$, $\eta_2 \in \mathcal{P}_{m_0}(\Gamma)$ be constrained MFG equilibria, such that $m_1(t) =e_t \sharp \eta_1$ , $m_2(t)= e_t \sharp \eta_2$ and let $u_1$, $u_2$ be the value functions of $J_{\eta_1}$ and $J_{\eta_2}$, respectively. Let $x_0\in\overline{\Omega}$ and let $\gamma$ be an optimal trajectory for $u_1$ at $(0,x_0)$. We get
	\begin{align*}
	u_1(0,x_0)&=\int_0^{T} \Big[L(\gamma(s),\dot \gamma(s))+ F(\gamma(s),  m_1(s))\Big]\ ds +G(\gamma(T),m_1(T)),\\
	u_2(0,x_0)&\leq \int_0^{T} \Big[L(\gamma(s),\dot \gamma(s))+ F(\gamma(s),  m_2(s))\Big]\ ds+G(\gamma(T),m_2(T)).
	\end{align*}
	Therefore, 
	\begin{align*}
	&G(\gamma(T),m_1(T))-G(\gamma(T),m_2(T))\leq u_1(0,x_0)-u_2(0,x_0)\\
	&-\int_0^{T} \Big[L(\gamma(s),\dot \gamma(s))+ F(\gamma(s),  m_1(s))\Big]\ ds + \int_0^{T}\Big[L(\gamma(s),\dot \gamma(s))+ F(\gamma(s),  m_2(s))\Big]\ ds\\
	&=u_1(0,x_0))-u_2(0,x_0)+\int_0^{T} F(\gamma(s),  m_2(s))-F(\gamma(s),  m_1(s))\ ds.
	\end{align*}
	Recalling that $supp(\eta_1)\subset \Gamma^{\eta_1}[x_0]$, we integrate on $\Gamma$ respect to $d\eta_1$ to obtain
	\begin{align*}
	&\int_{\Gamma} \Big[G(\gamma(T),m_1(T))-G(T,m_2(T))\Big] d\eta_1(\gamma)\leq\\
	&\leq\int_{\Gamma}\Big[u_1(0,\gamma(0))-u_2(0,\gamma(0))\Big] d\eta_1(\gamma) +\int_{\Gamma}\int_0^{T}  \Big[F(\gamma(s),  m_2(s))-F(\gamma(s),  m_1(s))\Big]\,ds\,d\eta_1(\gamma).
	\end{align*}
	Recalling the definition of $e_t$ and using the change of variables $e_t(\gamma)=x$ in the above inequality, we get
	\begin{align*}
	\int_{\Gamma}\Big[ G(\overbrace{\gamma(T)}^{e_{T}(\gamma)},m_1(T))-G(\overbrace{\gamma(T)}^{e_{T}(\gamma)},m_2(T))\Big] d\eta_1(\gamma)&=\int_{\overline\Omega}\Big[ G(x,m_1(T))-G(x,m_2(T))\Big]\, m_1(T,dx),\\
	\int_{\Gamma}\Big[u_1(0,\overbrace{\gamma(0)}^{e_{0}(\gamma)})-u_2(0,\overbrace{\gamma(0)}^{e_{0}(\gamma)})\Big] d\eta_1(\gamma)&=\int_{\overline\Omega} \Big[u_1(0,x)-u_2(0,x)\Big]\, m_1(0,dx),\\
	\int_0^{T}\int_{\Gamma}  \Big[F(\overbrace{\gamma(s)}^{e_{s}(\gamma)},  m_2(s))-F(\overbrace{\gamma(s)}^{e_{s}(\gamma)},  m_1(s))\Big]\,d\eta_1(\gamma)\,ds&=\int_0^{T}\int_{\overline\Omega}  \Big[F(x,  m_2(s))-F(x,  m_1(s))\Big]\,m_1(s,dx)\,ds.
	\end{align*}
	So, we deduce that
	\begin{align}\label{5b}
	&\int_{\overline\Omega} \Big[G(x,m_1(T))-G(x,m_2(T))\Big]\, m_1(T,dx)\\
	&\leq\int_{\overline\Omega} \Big[u_1(0,x)-u_2(0,x)\Big]\, m_1(0,dx)+\int_0^{T}\int_{\overline\Omega} \Big[ F(x,  m_2(s))-F(x,  m_1(s))\,m_1(s,dx)\Big]\,ds.\nonumber
	\end{align}
	In a similar way, we obtain 
	\begin{align}\label{5bis}
	&\int_{\overline\Omega} \Big[G(x,m_2(T))-G(x,m_1(T))\Big]\, m_2(T,dx)\\
	&\leq \int_{\overline\Omega} \Big[u_2(0,x)-u_1(0,x)\Big]\, m_2(0,dx)+\int_0^{T}\int_{\overline\Omega}  \Big[F(x,m_1(s))-F(x,m_2(s))\Big]\,m_2(s,dx)\,ds.\nonumber
	\end{align}
	Summing the inequalities (\ref{5b}) and (\ref{5bis}), we deduce that
	\begin{align*}
	&\int_{\overline\Omega} [G(x,m_1(T))-G(x,m_2(T))]\,(m_1(T,dx)-m_2(T,dx))\\
	&\leq \int_{\overline\Omega}[u_1(0,x)-u_2(0,x)]\,(m_1(0,dx)-m_2(0,dx))\\
	&+\int_0^{T} \int_{\overline\Omega} \Big[F(x,m_2(s))-F(x,  m_1(s))\Big]\,(m_1(s,dx)-m_2(s,dx)) \,ds
	\end{align*}
	Since $m_1(0,dx)=m_2(0,dx)=m_0$, by using the monotonicity assumption on $G$ and $F$, we obtain that 
	\begin{align*}
	&0\geq\int_0^{T} \int_{\overline\Omega} \Big[F(x,m_2(s))-F(x,m_1(s))\Big]\,(m_1(s,dx)-m_2(s,dx)) \,ds\geq\\
	& \int_{\overline\Omega} \Big[G(x,m_1(T))-G(x,m_2(T))\Big]\,(m_1(T,dx)-m_2(T,dx))\geq 0.
	\end{align*}
	Therefore,
	\begin{equation*}\label{uc}
	\int_{\overline\Omega} \Big[F(x,m_2(s))-F(x,  m_1(s))\Big]\,(m_1(s,dx)-m_2(s,dx))=0 \ \ \  \mbox{a.e.} \ s \in [0,T],
	\end{equation*}
	and also
	\begin{equation*}
	\int_{\overline\Omega} \Big[G(x,m_1(T))-G(x,  m_2(T))\Big]\,(m_1(T,dx)-m_2(T,dx))=0.
	\end{equation*}
	Thus, by the strict monotonicity of $F$ and $G$, we conclude that $F(x,m_1)=F(x,m_2)$ for all $x \in \overline{\Omega}$ and $G(x,m_1)=G(x,m_2)$ for all $x\in\overline{\Omega}$. Consequently, we have that $J_{\eta_1}$ is equal to $J_{\eta_2}$.
\end{proof}
\begin{remark}
Suppose that $G$ satisfies \eqref{g} and $F$ satisfies the following condition
\begin{equation*}
\int_{\overline \Omega} \Big[F(x,m_1)-F(x,m_2)\Big]d(m_1-m_2)(x)\ >\ 0,
\end{equation*}
for any $m_1,m_2\in {\mathcal P}(\overline \Omega)$ with $m_1\neq m_2$. Then, proceeding as in the proof of Theorem \ref{u}, one can show that the mild solution $(u,m)$ of the constrained MFG problem in $\overline\Omega$ is unique. 
\end{remark}

\end{document}